\documentclass[12pt,reqno]{amsart}
\usepackage{enumerate, latexsym, amsmath, amsfonts, amssymb, amsthm, color}
\def\pmod #1{\ ({\rm{mod}}\ #1)}
\def\Z{\mathbb Z}

\def\bg{\bigg}
\def\({\bg(}
\def\){\bg)}

\def\Ack{\medskip\noindent {\bf Acknowledgments}}
\theoremstyle{plain}
\newtheorem{theorem}{Theorem}

\newtheorem{lemma}{Lemma}

\theoremstyle{definition}

\theoremstyle{remark}
\newtheorem{remark}{Remark}

 \vspace{4mm}

\begin{document}

\hbox{}

\title
{On the divisibility of some truncated hypergeometric series}

\author
{Guo-Shuai Mao}

\address {Department of Mathematics, Nanjing
University, Nanjing 210093, People's Republic of China}
\email{mg1421007@smail.nju.edu.cn}

\author{Hao Pan}
\address {Department of Mathematics, Nanjing
University, Nanjing 210093, People's Republic of China}
\email{haopan79@zoho.com}

\keywords{central binomial coefficients, congruences.
\newline \indent 2010 {\it Mathematics Subject Classification}. Primary 11B65; Secondary 05A10, 11A07.
\newline \indent The first author is supported by the Natural Science Foundation of China (grant 11571162).}

\begin{abstract} Let $p$ be an odd prime and $r\geq 1$. Suppose that $\alpha$ is a $p$-adic integer with $\alpha\equiv2a\pmod p$ for some $1\leq a<(p+r)/(2r+1)$. We confirm a conjecture of Sun and prove that
$${}_{2r+1}F_{2r}\bigg[\begin{matrix}\alpha&\alpha&\ldots&\alpha\\
&1&\ldots&1\end{matrix}\bigg|\,1\bigg]_{p-1}\equiv0\pmod{p^2},$$
where the truncated hypergeometric series
$$
{}_{q+1}F_{q}\bigg[\begin{matrix}x_0&x_1&\ldots&x_{q}\\
&y_1&\ldots&y_q\end{matrix}\bigg|\,z\bigg]_{n}:=\sum_{k=0}^n\frac{(x_0)_k(x_1)_k\cdots(x_q)_k}{(y_1)_k\cdot (y_q)_k}\cdot\frac{z^k}{k!}.
$$
\end{abstract}
\maketitle

\section{Introduction}
\setcounter{lemma}{0}
\setcounter{theorem}{0}
\setcounter{corollary}{0}
\setcounter{remark}{0}
\setcounter{equation}{0}

For an odd prime $p$, let $\Z_p$ denote the ring of all $p$-adic integers. 
In \cite[Theorem 1.3 (i)]{SunZW12}, Sun  proved that  if $x\in\Z_p$ and $x\equiv -2a\pmod p$ for some $1\leq a\leq (p-1)/3$, then
\begin{equation}\label{Sunr3}
\sum_{k=0}^{p-1}(-1)^k\binom{x}{k}^3\equiv 0\pmod{p^3}.
\end{equation}
Motivated by (\ref{Sunr3}),  Sun conjectured \cite[Conjecture 4.4]{SunZW12} that for any integer $r\geq 2$ and odd prime $p$, if $x\in\Z_p$ and $x\equiv -2a\pmod p$ for some $1\leq a\leq (p+1)/(2r+1)$, then
\begin{equation}\label{SunConj}
\sum_{k=0}^{p-1}(-1)^k\binom{x}{k}^{2r+1}\equiv 0\pmod{p^2}.
\end{equation}

Define the truncated hypergeometric series
$$
{}_{r+1}F_{r}\bigg[\begin{matrix}x_0&x_1&\ldots&x_{r}\\
&y_1&\ldots&y_r\end{matrix}\bigg|\,z\bigg]_{n}=\sum_{k=0}^n\frac{(x_0)_k(x_1)_k\cdots(x_r)_k}{(y_1)_k\cdot (y_r)_k}\cdot\frac{z^k}{k!},
$$
where
$$
(x)_k=\begin{cases}
x(x+1)\cdots(x+k-1),&\text{if }k\geq 1,\\
1,&\text{if }k=0.
\end{cases}
$$
In fact, the truncated hypergeometric series is just the sum of the first finite terms of the original hypergeometric series. 
In the recent years, the arithmetic properties of the truncated hypergeometric series have been widely investigated (cf. \cite{AhOn00,DFLST16,He17,Liu17,LoRa16,Mo04,Mo05,OsSc09,OsStZu,SunZH13,SunZW13,Ta12}).
Note that
$$
\binom{x}{k}=(-1)^k\cdot\frac{(-x)_k}{(1)_k}
.$$ Clearly the left side of (\ref{SunConj}) coincides with
$$
{}_{2r+1}F_{2r}\bigg[\begin{matrix}-x&-x&\ldots&-x\\
&1&\ldots&1\end{matrix}\bigg|\,1\bigg]_{p-1}.
$$

In this short note, we shall give an affirmative answer to Sun's conjecture.
\begin{theorem}\label{main} 
Let $p$ be an odd prime and $r\geq 1$. 
Suppose that $\alpha\in\Z_p$ and $\alpha\equiv 2a\pmod p$ for some $1\leq a<(p+r)/(2r+1)$.
 Then
\begin{equation}\label{SunTH}
{}_{2r+1}F_{2r}\bigg[\begin{matrix}\alpha&\alpha&\ldots&\alpha\\
&1&\ldots&1\end{matrix}\bigg|\,1\bigg]_{p-1}\equiv 0\pmod{p^2}.
\end{equation}
\end{theorem}
Notice that the permitted range of $\alpha$ in Theorem \ref{main} is a little larger than the one conjectured by Sun.

Furthermore, we mention that Theorem \ref{main} (or just (\ref{Sunr3})) also implies another conjecture of Sun.
In \cite[Remark 1.2]{SunZW11}, Sun conjectured that for any prime $p$ with $p\equiv1\pmod{4}$,
\begin{equation}\label{SunConjCat}
\sum_{k=0}^{\frac{1}2(p-1)}\frac{C_k^{3}}{64^{k}}\equiv8\pmod{p^2},
\end{equation}
where 
$$
C_k:=\frac1{k+1}\binom{2k}{k}
$$
is the $k$-th Catalan number. It is easy to check that
$$
\frac{C_k}{4^k}=\frac{1}{k+1}\cdot\frac{(\frac12)_k}{(1)_k}=-2\cdot\frac{(-\frac12)_{k+1}}{(1)_{k+1}}.
$$
We also have $C_k\equiv 0\pmod{p}$ for each $(p+1)/2\leq k\leq p-2$. So (\ref{SunConjCat}) is actually equivalent to
\begin{equation}\label{3F212}
\sum_{k=0}^{p-1}\frac{(-\frac12)_k^3}{(1)_k^3}\equiv 0\pmod{p^2}.
\end{equation}
When $p\equiv 1\pmod{4}$, 
$$
-\frac12\equiv\frac{p-1}{2}=2\cdot\frac{p-1}{4}\pmod{p}.
$$
Thus since $(p-1)/4<(p+1)/3$, (\ref{3F212}) immediately follows from Theorem \ref{main} by setting $\alpha=-1/2$ and $r=1$.

Our proof of Theorem \ref{main}, which will be given in the subsequent section,  follows  a similar way in \cite{MP}. We shall construct a polynomial $\psi(x)\in\Z_p[x]$ with $\psi(p)=0$ such that
$$
{}_{2r+1}F_{2r}\bigg[\begin{matrix}\alpha&\alpha&\ldots&\alpha\\
&1&\ldots&1\end{matrix}\bigg|\,1\bigg]_{p-1}=\psi(sp)
$$
for some $s\in\Z_p$. Then the proof of (\ref{SunTH}) can be easily reduced to show that $\psi'(0)$ is divisible by $p$.

\section{Proof of Theorem \ref{main}}
 \setcounter{lemma}{0}
\setcounter{theorem}{0}
\setcounter{corollary}{0}
\setcounter{remark}{0}
\setcounter{equation}{0}
\setcounter{conjecture}{0}

First, let us introduce several auxiliary lemmas.
\begin{lemma}\label{KM}
Let $m_1,\ldots,m_r$ be non-negative integers and $a$ be a positive integer. If $a>m_1+\ldots+m_r$, then
\begin{equation}\label{KMe}
{}_{r+1}F_{r}\bigg[\begin{matrix}-a&1+m_1&\ldots&1+m_r\\
&1&\ldots&1\end{matrix}\bigg|\,1\bigg]=0.
\end{equation}
\end{lemma}
\begin{proof}
This is a consequence of the Karlsson-Minton formula \cite[(12)]{Ka71}.
\end{proof}
\begin{lemma}\label{r1Frm1} Suppose that $m$ is a positive odd integer and $r\geq 0$ is even. Then
\begin{equation}\label{r1Frm1e}
{}_{r+1}F_{r}\bigg[\begin{matrix}-m&-m&\ldots&-m\\
&1&\ldots&1\end{matrix}\bigg|\,1\bigg]=0.
\end{equation}
\end{lemma}
\begin{proof}
We have
$$
{}_{r+1}F_{r}\bigg[\begin{matrix}-m&-m&\ldots&-m\\
&1&\ldots&1\end{matrix}\bigg|\,1\bigg]=\sum_{k=0}^m\frac{(-m)_k^{r+1}}{(1)_k^{r+1}}=
\sum_{k=0}^m(-1)^k\binom{m}{k}^{r+1}.
$$
Then (\ref{r1Frm1e}) follows from the fact
$$
\sum_{k=0}^m(-1)^k\binom{m}{k}^{r+1}=\sum_{k=0}^m(-1)^{m-k}\binom{m}{m-k}^{r+1}=
-\sum_{k=0}^m(-1)^{k}\binom{m}{k}^{r+1}.
$$
\end{proof}
Define the $n$-th harmonic number
$$
H_n:=\sum_{k=1}^n\frac1k.
$$
In particular, we set $H_0=0$.
\begin{lemma}\label{2a2r1HkH2ak1} Suppose that $p$ is an odd prime and $0\leq a<p/2$ is an integer. 
Then for any $r\geq 0$,  
\begin{equation}
\sum_{k=0}^{p-2a}\frac{(2a)_k^{2r+1}}{(1)_k^{2r+1}}\cdot H_k\equiv
-\sum_{k=0}^{p-2a}\frac{(2a)_k^{2r+1}}{(1)_k^{2r+1}}\cdot H_{2a+k-1}\pmod{p}.
\end{equation}
\end{lemma}
\begin{proof} Clearly
\begin{align*}
\sum_{k=0}^{p-2a}\frac{(2a)_k^{2r+1}}{(1)_k^{2r+1}}\cdot H_k\equiv&
\sum_{k=0}^{p-2a}\frac{(2a-p)_k^{2r+1}}{(1)_k^{2r+1}}\cdot H_k
=\sum_{k=0}^{p-2a}(-1)^kH_k\binom{p-2a}k^{2r+1}\\
=&\sum_{k=0}^{p-2a}(-1)^{p-2a-k}H_{p-2a-k}\binom{p-2a}{p-2a-k}^{2r+1}\\
\equiv&-\sum_{k=0}^{p-2a}(-1)^{k}H_{2a+k-1}\binom{p-2a}{k}^{2r+1}\pmod{p},
\end{align*}
where in the last step we use the well-known fact
$$
H_{k}\equiv H_{p-1-k}\pmod{p}
$$
for any $0\leq k\leq p-1$.
\end{proof}
Now we are ready to prove Theorem \ref{main}.
Let
$$\psi(x)={}_{2r+1}F_{2r}\bigg[\begin{matrix}2a-x&2a-x&\ldots&2a-x\\
&1&\ldots&1\end{matrix}\bigg|\,1\bigg]_{p-1}.$$
Then by Lemma \ref{r1Frm1}, 
$$
\psi(p)={}_{2r+1}F_{2r}\bigg[\begin{matrix}2a-p&2a-p&\ldots&2a-p\\
&1&\ldots&1\end{matrix}\bigg|\,1\bigg]=0.
$$ 
Furthermore, for any integers $s_0,s_1,\ldots,s_r,t_1,\ldots,t_r$, evidently we have
\begin{align}\label{2r1F2rstp}
&{}_{2r+1}F_{2r}\bigg[\begin{matrix}2a-s_0p&2a-s_1p&\ldots&2a-s_rp\\
&1+t_1p&\ldots&1+t_rp\end{matrix}\bigg|\,1\bigg]_{p-1}\notag\\
\equiv&
{}_{2r+1}F_{2r}\bigg[\begin{matrix}2a-p&2a-p&\ldots&2a-p\\
&1&\ldots&1\end{matrix}\bigg|\,1\bigg]_{p-1}=0\pmod{p}.
\end{align}

Let $s=(2a-\alpha)/p$. By the Taylor expansion of $\psi(x)$, we have
\begin{align*}
{}_{2r+1}F_{2r}\bigg[\begin{matrix}\alpha&\alpha&\ldots&\alpha\\
&1&\ldots&1\end{matrix}\bigg|\,1\bigg]_{p-1}=\psi(sp)\equiv&\psi(p)+(s-1)p\cdot\psi'(p)\\
\equiv&(s-1)p\cdot\psi'(0)\pmod{p^2}.
\end{align*}
It suffices to show that $\psi'(0)$ is divisible by $p$.
Let$$\phi(x)={}_{2r+1}F_{2r}\bigg[\begin{matrix}2a-x&2a&\ldots&2a\\
&1&\ldots&1\end{matrix}\bigg|\,1\bigg]_{p-1}.$$
Then
\begin{equation}\label{psi2r1phi}
\psi'(x)=(2r+1)\phi'(x).
\end{equation}
If $2r+1\equiv 0\pmod{p}$, clearly $\psi'(0)\equiv 0\pmod{p}$. So we may assume that $2r+1$ is not divisible by $p$. Below we only need to
prove that
\begin{equation}
\phi'(0)\equiv0\pmod{p}.
\end{equation}

Our strategy is to compute $\phi'(0)$ modulo $p$ in two different ways.
One way is simple. Clearly
\begin{equation}\label{difffactoral}
\frac{d((m-x)_k)}{dx}\bigg|_{x=0}=-(m)_k\sum_{i=0}^{k-1}\frac1{m+i}=(m)_k(H_{m-1}-H_{m+k-1})
\end{equation}
for any positive integer $m$.
So we get
\begin{align}\label{phi0p}
\phi'(0)=&\sum_{k=0}^{p-1}\frac{(2a)_k^{2r}}{(1)_k^{2r+1}}\cdot\frac{d((2a-x)_k)}{dx}\bigg|_{x=0}\notag\\
=&\sum_{k=0}^{p-1}\frac{(2a)_k^{2r+1}}{(1)_k^{2r+1}}\cdot(H_{2a-1}-H_{2a+k-1})\notag\\
\equiv&-\sum_{k=0}^{p-2a}\frac{(2a)_k^{2r+1}}{(1)_k^{2r+1}}\cdot H_{2a+k-1}\pmod p,
\end{align}
by noting that $p$ divides $(2a)_k$ for $2a+1\leq k\leq p-1$.

However, the other way is a little complicated. 
According to Lemma \ref{KM}, 
we have
\begin{align}\label{phi2psum}
\phi(2p)=&{}_{2r+1}F_{2r}\bigg[\begin{matrix}2a-2p&2a&\ldots&2a\\
&1&\ldots&1\end{matrix}\bigg|\,1\bigg]_{p-1}\notag\\
=&{}_{2r+1}F_{2r}\bigg[\begin{matrix}2a-2p&2a&\ldots&2a\\
&1&\ldots&1\end{matrix}\bigg|\,1\bigg]-\sum_{k=0}^{p-2a}\frac{(2a-2p)_{p+k}(2a)_{p+k}^{2r}}{(1)_{p+k}^{2r+1}}\notag\\
=&-\sum_{k=0}^{p-2a}\frac{(2a-2p)_{p+k}(2a)_{p+k}^{2r}}{(1)_{p+k}^{2r+1}},
\end{align}
since
$
2p-2a>2r(2a-1)
$ now.
It is easy to check that
$$
\frac{(2a-2p)_p(2a)_p^{2r}}{(1)_p^{2r+1}}\equiv-1\pmod{p}.
$$
Then we get
\begin{align}\label{2a2p2a1pk}
\sum_{k=0}^{p-2a}\frac{(2a-2p)_{p+k}(2a)_{p+k}^{2r}}{(1)_{p+k}^{2r+1}}=&
\frac{(2a-2p)_p(2a)_p^{2r}}{(1)_p^{2r+1}}\sum_{k=0}^{p-2a}\frac{(2a-p)_{k}(2a+p)_{k}^{2r}}{(1+p)_{k}^{2r+1}}\notag\\
\equiv&-\sum_{k=0}^{p-2a}\frac{(2a-p)_{k}(2a+p)_{k}^{2r}}{(1+p)_{k}^{2r+1}}
\pmod{p^2},
\end{align}
since $p$ divides the last sum by (\ref{2r1F2rstp}).

While in view of (\ref{difffactoral}), for any $t\in\Z_p$, we have
\begin{align*}
(2a-tp)_k-(2a)_k\equiv&tp\cdot \frac{d((2a-x)_k)}{dx}\bigg|_{x=0}\\
\equiv&
tp\cdot (2a)_k(H_{2a-1}-H_{2a+k-1})\pmod{p^2}.
\end{align*}
It follows that
\begin{align}\label{2ap2ap1p2a1}
&\sum_{k=0}^{p-2a}\frac{(2a-p)_{k}(2a+p)_{k}^{2r}}{(1+p)_{k}^{2r+1}}-\sum_{k=0}^{p-2a}\frac{(2a)_{k}^{2r+1}}{(1)_{k}^{2r+1}}\notag\\
\equiv&p\sum_{k=0}^{p-2k}\frac{(2a)_{k}^{2r+1}}{(1)_{k}^{2r+1}}\cdot
\big((1-2r)\cdot (H_{2a-1}-H_{2a+k-1})-(2r+1)H_k\big)\notag\\
\equiv&p\sum_{k=0}^{p-2k}\frac{(2a)_{k}^{2r+1}}{(1)_{k}^{2r+1}}\cdot
\big((2r-1)H_{2a+k-1}-(2r+1)H_k\big)
\pmod{p^2},
\end{align}
where (\ref{2r1F2rstp}) is applied again in the last step.

Finally, combining (\ref{2ap2ap1p2a1}) with (\ref{phi2psum}) and (\ref{2a2p2a1pk}), and applying Lemma \ref{2a2r1HkH2ak1}, we obtain that
\begin{align}\label{phi0p2r}
\phi'(0)\equiv&\frac{\phi(2p)-\phi(0)}{2p}\notag\\
\equiv&\frac12\sum_{k=0}^{p-2k}\frac{(2a)_{k}^{2r+1}}{(1)_{k}^{2r+1}}\cdot
\big((2r-1)H_{2a+k-1}-(2r+1)H_k\big)\notag\\
\equiv&2r\sum_{k=0}^{p-2k}\frac{(2a)_{k}^{2r+1}}{(1)_{k}^{2r+1}}\cdot
H_{2a+k-1}\pmod{p}.
\end{align}
Recall that $2r+1$ has been assumed to be co-prime to $p$. Hence by (\ref{phi0p}) and (\ref{phi0p2r}), we must have
$$
\phi'(0)\equiv 0\pmod{p}.
$$
All are done.\qed

\begin{remark}
In view of (\ref{psi2r1phi}), we always have $\psi'(0)\equiv0\pmod{p}$ whenever $p$ divides $2r+1$. So if $\alpha\in\Z_p$ and $\alpha\equiv 2a\pmod{p}$ for some $1\leq a\leq (p-1)/2$, then
\begin{equation}
{}_{np}F_{np-1}\bigg[\begin{matrix}\alpha&\alpha&\ldots&\alpha\\
&1&\ldots&1\end{matrix}\bigg|\,1\bigg]_{p-1}\equiv 0\pmod{p^2}
\end{equation}
for any odd $n\geq 1$.
\end{remark}
\begin{remark}
Suppose that $r\geq 1$ is not divisible by $(p-1)/2$. Let $r_*\geq 1$ be the least positive residue of $r$ modulo $(p-1)/2$. Then by the Fermat little theorem,
\begin{align*}
\frac{d}{dx}\bigg(\sum_{k=0}^{p-1}\frac{(2a-x)_k(2a)_k^{2r}}{(1)_k^{2r+1}}\bigg)\bigg|_{x=0}
\equiv\frac{d}{dx}\bigg(\sum_{k=0}^{p-1}\frac{(2a-x)_k(2a)_k^{2r_*}}{(1)_k^{2r_*+1}}\bigg)\bigg|_{x=0}\pmod{p}.
\end{align*}
It follows that for any $p$-adic integer $\alpha$ with $\alpha\equiv2a\pmod{p}$ for some $1\leq a<(p+r_*)/(2r_*+1)$,
\begin{equation}
{}_{2r+1}F_{2r}\bigg[\begin{matrix}\alpha&\alpha&\ldots&\alpha\\
&1&\ldots&1\end{matrix}\bigg|\,1\bigg]_{p-1}\equiv 0\pmod{p^2}
\end{equation}
\end{remark}
 
\medskip
\Ack. The authors are grateful to Professor Zhi-Wei Sun for his very helpful comments on this paper.

\setcounter{conjecture}{0} \end{document}